\documentclass[11pt]{article}
\PassOptionsToPackage{obeyspaces}{url}
\usepackage{amsfonts, amsmath, amssymb}

\usepackage[hidelinks]{hyperref}
\usepackage{tabularx}
\usepackage{bookmark}
\bookmarksetup{open,numbered,depth=5} 

\usepackage{amsthm}

\usepackage{breakurl}
\usepackage{tikz}

\usepackage{etoolbox} 
\patchcmd{\thebibliography}{\leftmargin\labelwidth}{\leftmargin\labelwidth\addtolength\itemsep{-0.1\baselineskip}}{}{}

\usepackage{epstopdf}

\usepackage[justification=raggedright]{caption}

\newcommand*\samethanks[1][\value{footnote}]{\footnotemark[#1]}

\author{Miguel Benitez\thanks{Department of Mathematical Sciences, Carnegie Mellon University, Pittsburgh, PA 15213, USA\@.} \qquad Siran Chen\samethanks\qquad Tianhui Han\samethanks\\ R. Amzi Jeffs\samethanks\, \thanks{Supported by the National Science Foundation through Award No. 2103206.}  \qquad Kinapal Paguyo\samethanks[1] \qquad Kevin A. Zhou\samethanks[1]}

\title{
\vspace{-1em}Realizing convex codes with axis-parallel boxes}
\date{March 2023}

\usepackage[nameinlink]{cleveref}

\newtheorem{theorem}{Theorem}
\newtheorem{lemma}[theorem]{Lemma}
\newtheorem{corollary}[theorem]{Corollary}
\newtheorem{proposition}[theorem]{Proposition}

\theoremstyle{remark}
\newtheorem{definition}[theorem]{Definition}
\newtheorem{example}[theorem]{Example}
\newtheorem{remark}[theorem]{Remark}

\newcommand*{\eqdef}{\stackrel{\mbox{\normalfont\tiny def}}{=}}  
\newcommand*{\R}{\mathbb{R}}                                     
\newcommand*{\C}{\mathcal{C}}                                     
\newcommand*{\D}{\mathcal{D}}                                     
\newcommand*{\U}{\mathcal{U}}                                     
\newcommand*{\V}{\mathcal{V}}                                     
\newcommand*{\W}{\mathcal{W}}                                     
\newcommand*{\I}{\mathcal{I}}                                     
\newcommand*{\J}{\mathcal{J}}                                     
\DeclareMathOperator{\code}{code}                                 
\DeclareMathOperator{\odim}{odim}                                 
\DeclareMathOperator{\cdim}{cdim}                                 
\DeclareMathOperator{\bdim}{bdim}                                 
\DeclareMathOperator{\patt}{patt}                                 

 \newcommand*{\intprod}{\hspace{0.15em}\includegraphics{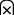}\hspace{0.15em}}

 \usepackage[inline]{asymptote}

\begin{document}
\maketitle
\vspace{-2.5em}
\begin{abstract}
    \vspace{-0.5em} Every ordered collection of sets in Euclidean space can be associated to a combinatorial code, which records the regions cut out by the sets in space. 
    Given two ordered collections of sets, one can form a third collection in which the $i$-th set is the Cartesian product of the corresponding sets from the original collections.
    We prove a general ``product theorem" which characterizes the code associated to the collection resulting from this operation, in terms of the codes associated to the original collections. 
    We use this theorem to characterize the codes realizable by axis-parallel boxes, and exhibit differences between this class of codes and those realizable by convex open or closed sets.
    We also use our theorem to prove that a ``monotonicity of open convexity" result of Cruz, Giusti, Itskov, and Kronholm holds for closed sets when some assumptions are slightly weakened. 
\end{abstract}

\vspace{-1em}\section{Introduction} \label{sec:introduction}\vspace{-0.8em}

Given a collection $\U = \{U_1, \ldots, U_n\}$ of sets in $\R^d$, one can record how they intersect and cover each other using a combinatorial code on the index set $[n]$, defined as follows:
\begin{align*}
\code(\U) &\eqdef 
\{\sigma\subseteq [n]\, \vert\,  \text{there exists $p\in \R^d$ with $p\in U_i$ if and only if $i\in\sigma$}\}. 
\end{align*}
We say that $\U$ is a \emph{realization} of $\code(\U)$, and elements of a code are called \emph{codewords}.
A codeword is \emph{maximal} if it is not properly contained in any other codeword. 

For notational convenience, we define the \emph{intersection pattern} of $\U$ at a point $p$ to be the set $\patt_\U(p) \eqdef \{i\in[n] \mid p\in U_i\}$.
For $\sigma\subseteq [n]$ the \emph{atom} of $\sigma$ in $\U$ is $\U^\sigma \eqdef \{p\in \R^d \mid \patt_\U(p) = \sigma\}$. 
With this notation, $\code(\U)$ consists of the sets $\sigma\subseteq [n]$ with nonempty atoms.
\Cref{fig:intro-example} shows a realization of a code by four axis-parallel rectangles in $\R^2$. 
Throughout, we write maximal codewords in bold.

\begin{figure}[h]
    \[
   \includegraphics{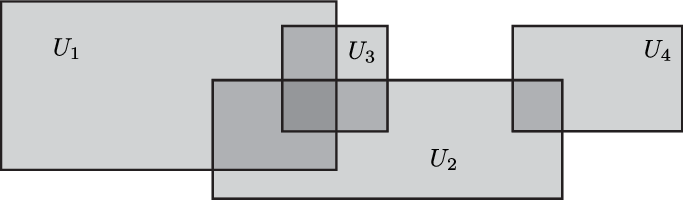}
    \]
    \caption{Four axis-parallel rectangles in $\R^2$ which realize the code $\{\mathbf{123}, \mathbf{24}, 12, 13, 23, 1, 2, 3, 4, \emptyset\}$.}
    \label{fig:intro-example}
\end{figure}

\paragraph{Convex codes.}
Given a collection of sets, it is straightforward to compute its associated code.
The inverse problem---finding a collection of sets that realize a given code---may be complicated, especially when geometric constraints are placed on the sets in the realization.
In 2013, Curto, Itskov, Veliz\nobreakdash-Cuba, and Youngs \cite{CIVY} initiated the study of \emph{open convex codes}, which are codes with realizations consisting of convex open sets.
One may also consider \emph{closed convex codes}, which have realizations by closed convex sets. 
Curto, Itskov, Veliz-Cuba, and Youngs were motivated by neuroscientific phenomena, and the study of open and closed convex codes additionally generalizes the more classical study of $d$-representable complexes (see Tancer's survey \cite{tancer}).
Since its inception the study of open and closed convex codes has thus led to rich results of purely mathematical interest, such as new discrete geometry theorems, families of extremal examples, and more. 

\begin{remark}
Consistent with the convex codes literature, we only consider bounded sets in $\R^d$. This is equivalent to considering codes that contain $\emptyset$ as a codeword (if any sets are unbounded, we may intersect them with a large open or closed ball to obtain a bounded realization of the same code). 
\end{remark}

Beyond determining whether a code has an open or closed convex realization, it is natural to seek  realizations that are minimal in dimension. 
To this end, it is useful to define the \emph{open} (respectively, \emph{closed}) \emph{embedding dimension} of a code.
These are
\begin{align*}
    \odim(\C) &\eqdef \min\{d\mid \text{$\C$ has a realization by open convex sets in $\R^d$}\}, \text{ and}\\
    \cdim(\C) &\eqdef \min\{d\mid \text{$\C$ has a realization by closed convex sets in $\R^d$}\}.
\end{align*}
By convention, the minimum over the empty set is $\infty$. 

A complete classification of open or closed convex codes and their embedding dimensions is currently out of reach, and an efficient classification seems to be out of the question: results of Kunin, Lienkaemper, and Rosen \cite{KLR} posit that recognizing open convex codes in $\R^2$ is at least NP-hard. 
A further difficulty is the difference between open convex and closed convex realizations---
Lienkaemper, Shiu, and Woodstock \cite{LSW} described a code which has a closed convex realization, but no open convex realization, and
Cruz, Giusti, Itskov, and Kronholm \cite{CGIK} described a code with the opposite behavior.
Families of codes which exhibit arbitrary disparities between closed and open embedding dimension were given more recently by Jeffs \cite{jeffs_embedding_vectors}, but general criteria for understanding and classifying these differences are lacking. 

\paragraph{Box codes.}
Given the difficulties of classifying open and closed convex codes, we tackle the simpler problem of which codes can be realized by axis-parallel boxes. 
An \emph{axis-parallel box} in $\R^d$ is either the empty set, or the Cartesian product of nonempty intervals $[a_1, b_1]\times [a_2,b_2]\times \cdots \times [a_d,b_d]$. 
For brevity, we use the term \emph{box} from here on.
Note that we are considering closed boxes, and we are allowing for boxes that are not full-dimensional---in particular, a point is a box. 
 We say that a code is a \emph{box code} or \emph{box-convex} if it has a realization in $\R^d$ in which every set is an axis-parallel box.
Lastly, we define the \emph{box embedding dimension} of a code $\C$ to be
\[
   \bdim(\C) \eqdef \min\{d\mid \text{$\C$ has a realization in $\R^d$ in which each set is a box}\}.
\]
Note immediately that $\cdim(\C)\le \bdim(\C)$, as every box is a closed convex set. 
Later results will also imply that $\odim(\C)\le \bdim(\C)$ (see \Cref{thm:integercoordinates}).  
\paragraph{Our results.}
Our main result is \Cref{thm:product}, which applies to any collections of sets, but proves particularly useful for analyzing box-convex codes and their embedding dimensions. 
To state the result we first require a definition.

\begin{definition}
The \emph{intersection product} of two codes $\C\subseteq 2^{[n]}$ and $\D\subseteq 2^{[n]}$ is the code
\[
    \C \intprod \D \eqdef \{ c_1 \cap c_2 \mid c_1\in \C \text{ and } c_2\in \D\}.
\]
\end{definition}

\begin{theorem}[Product theorem]\label{thm:product}
Let $\U = \{U_1, \ldots, U_n\}$ and $\V = \{V_1, \ldots, V_n\}$ be collections of (not necessarily open, closed, or convex) sets in $\R^{d_1}$ and $\R^{d_2}$ respectively.
Define $\W \eqdef \{W_1,\ldots, W_n\}$ where $W_i \eqdef U_i\times V_i$. 
Then
\[
\code(\W) = \code(\U) \intprod \code(\V). 
\]
\end{theorem}

Since boxes can be decomposed as Cartesian products of intervals, \Cref{thm:product} will let us reduce the study of box-convex codes to the study of \emph{interval codes} which have realizations in $\R^1$ in which every set is a closed interval.
We give two immediate corollaries of \Cref{thm:product} below.


\begin{corollary}\label{cor:dimensions}
For any codes $\C\subseteq 2^{[n]}$ and $\D\subseteq 2^{[n]}$, we have $\odim(\C\intprod \D) \le \odim(\C) + \odim(\D)$.
The analogous inequalities hold for $\cdim(\C\intprod \D)$ and $\bdim(\C\intprod \D)$. 
\end{corollary}
\begin{proof}
Note that if $U_i$ and $V_i$ are both open and convex, then so is $U_i \times V_i$. 
The same is true when the sets are closed and convex, or when they are boxes. 
Therefore, if $\U$ is an open (respectively, closed or box) convex realization of $\C$ in $\mathbb{R}^{d_1}$ and $\V$ is an open (respectively, closed or box) convex realization of $\D$ in $\R^{d_2}$, then the collection $\W$ of \Cref{thm:product} is a realization of $\C\intprod \D$ in $\R^{d_1+d_2}$ of the same type. 
\end{proof}

\begin{corollary}\label{cor:classification}
Let $\C\subseteq 2^{[n]}$ be a code.
Then $\C$ is box-convex in $\R^d$ if and only if $\C$ is the intersection product of $d$-many interval codes. 
\end{corollary}
\begin{proof}
Let $\mathcal B = \{B_1, B_2, \ldots , B_n \}$ be a realization of $\C$ by boxes in $\R^d$. 
Each set $B_i$ can represented as a Cartesian product of $d$-many intervals: $B_i = I_{i}^{(1)} \times I_{i}^{(2)} \times \cdots \times I_{i}^{(d)} $. 
By \Cref{thm:product}, we know that $\code(\mathcal B) =  \code(\mathcal I^{(1)}) \intprod \code(\mathcal I^{(2)}) \intprod \cdots \intprod \code(\mathcal I^{(d)})$ where $\mathcal I^{(j)} = \{I_1^{(j)}, I_2^{(j)}, \ldots, I_n^{(j)}\}$. 
Therefore, we can conclude that $\C$ is the intersection product of $d$-many interval codes.
The converse follows similarly from \Cref{thm:product}, as one may form a realization of $\C$ by taking Cartesian products of intervals that realize the interval codes whose intersection product is $\C$. 
\end{proof}

Note that a decomposition of a code $\C$ as an intersection product of interval codes need not be unique. 
An important feature of \Cref{cor:classification} is that it yields a brute-force algorithm for recognizing whether or not a code is box-convex in $\R^d$.
First, we enumerate all interval codes on $[n]$, which is possible because each such code is determined by the relative order of the $2n$ endpoints of the intervals in question (we explain this fact formally in \Cref{sec:integer-coords}). 
We then take all possible $d$-fold intersection products of these interval codes, and check whether or not $\C$ arises in this process. 

Although the algorithm described above is far from efficient, this situation already stands in contrast to open and closed convex codes---Bukh and Jeffs \cite{BJ22a} recently proved that open and closed convex codes realizable in $\R^2$ can be recognized algorithmically, but no such decidability result is known in dimension three or higher. 
It remains an open question whether or not there is an \emph{efficient} algorithm to recognize box-convex codes in $\R^d$. 

The remainder of the paper is devoted to proving \Cref{thm:product} in \Cref{sec:productproof}, and then exploring a variety of applications.
In \Cref{sec:small-codes} we classify all codes $\C\subseteq 2^{\{1,2,3\}}$ which are convex, but not box-convex---it turns out that there are only three such codes up to symmetry. 
We use one of these codes to show that a ``monotonicity of open convexity" result proved by Cruz, Giusti, Itskov, and Kronholm \cite{CGIK} does not hold for box-convexity. 
In \Cref{sec:embedding-dimension-gap} we describe a family of codes whose box embedding dimension ranges over all integer values greater than or equal to two, while the open and closed embedding dimensions remain fixed and equal to two. 
We prove in \Cref{sec:integer-coords} that codes $\C\subseteq 2^{[n]}$ realizable by boxes in $\R^d$ always have realizations in which the corners of the boxes have integer coordinates between 1 and $2n$, and additionally the interiors of the boxes realize the same code.
In particular, there is no difference between open and closed box codes. 
Finally, \Cref{sec:weak-monotonicity} applies \Cref{thm:product} beyond box codes by showing that a weak analog of the aforementioned ``monotonicity of open convexity" result holds for closed convex codes. 
The analog we prove is necessarily weak, as Gambacini, Jeffs, Macdonald, and Shiu \cite{GJMS} proved that the strong analog does not hold.


\section{Proof and examples of the product theorem (\texorpdfstring{\Cref{thm:product}}{Theorem 3})}\label{sec:productproof}

We start by proving \Cref{thm:product}, and then give two examples.
Let $\U$, $\V$, and $\W$ be as in the statement of \Cref{thm:product}. 
Observe that the nonempty atoms $\U^\sigma$ arising in the realization $\U$ cover all of $\R^{d_1}$, and likewise the nonempty atoms of the form $\V^\tau$ cover $\R^{d_2}$.
Therefore, all the Cartesian products $\U^\sigma \times \V^\tau$ as $\sigma$ and $\tau$ range over $\code(\U)$ and $\code(\V)$ respectively cover the entirety of $\R^{d_1+d_2}$. 
Fix codewords $\sigma \in \code(\U)$ and $\tau \in \code(\V)$.
Each point in $\U^\sigma\times \V^\tau$ can be written as a pair $(p,q)$ where $p\in \U^\sigma$ and $q\in \V^\tau$. 
We claim that $\patt_\W(p,q) = \sigma\cap \tau$. 

\paragraph{The inclusion $\patt_{\W}(p,q) \subseteq \sigma \cap \tau$.}
Let $k \in  \patt_{\W}(p,q)$.
Then $(p,q)\in W_k = U_k\times V_k$, which implies that $p\in U_k$ and $q\in V_k$. 
This in turn implies that $k\in \patt_\U(p) = \sigma$, and $k\in \patt_\V(q) = \tau$.
Thus $k\in \sigma\cap \tau$. 

\paragraph{The inclusion $\patt_{\W}(p,q) \supseteq \sigma \cap \tau$.} 
Let $k \in \sigma \cap \tau$.
Then $p\in U_k$ and $q\in V_k$.
Thus $(p,q)\in U_k\times V_k = W_k$, and so $k\in \patt_\W(p,q)$, as desired.

\paragraph{Summary.} 
Since $\patt_\W(p,q) = \sigma\cap \tau$ for all $(p,q)\in \U^\sigma\times\V^\tau$, we see that the codewords of $\code(\W)$ are exactly those of the form $\sigma\cap \tau$, as $\sigma$ ranges over $\code(\U)$ and $\tau$ ranges over $\code(\V)$.
This is equivalent to the statement that $\code(\W) = \code(\U) \intprod \code(\V)$, proving the result. \hfill \qed


\begin{figure}[h]
    \[
    \includegraphics{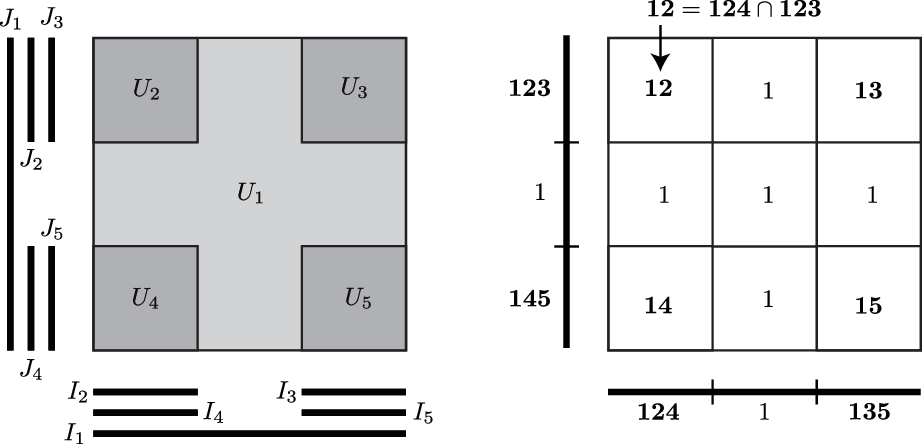}
    \]
    \caption{(Left) A realization of a box code in the plane, and its decomposition into interval codes in each coordinate direction. (Right) The atoms in the interval codes, and their products in the plane.}
    \label{fig:product-example-1}
\end{figure}

\begin{example} 
The left side of \Cref{fig:product-example-1} shows five boxes $\U = \{U_1, U_2, U_3, U_4, U_5\}$ in the plane, and their corresponding decompositions into intervals $\I = \{I_1, I_2, I_3, I_4, I_5\}$ and $\J = \{J_1, J_2,J_3,J_4,J_5\}$ in the $x$-direction and the $y$-direction respectively. 
We have $\code(\U) = \{\mathbf{12}, \mathbf{13}, \mathbf{14}, \mathbf{15}, 1, \emptyset\}$, which is the intersection product of $\code(\I) = \{\mathbf{124}, \mathbf{135}, 1, \emptyset\}$, and $\code(\J) = \{\mathbf{123}, \mathbf{145}, 1, \emptyset\}$, as guaranteed by \Cref{thm:product}. 
The right side of the figure shows the atoms of nonempty codewords in the $x$-direction and $y$-direction, and the Cartesian products of these atoms, which we used in the proof of \Cref{thm:product}.
For example, the top left square  $\I^{124} \times \J^{123}$ is contained in the atom $\U^{12}$ since $12 = 124\cap 123$.
\end{example}

\begin{example}
Recall the box-convex code $\C = \{\mathbf{123}, \mathbf{24}, 12, 13, 23, 1, 2, 3, 4, \emptyset\}$ from \Cref{fig:intro-example}. 
We can decompose each box in this realization as a product of two intervals. 
Letting $\I = \{I_1,I_2,I_3,I_4\}$ be the intervals in the $x$-direction and $\J=\{J_1,J_2,J_3,J_4\}$ be the intervals in the $y$-direction, we then have $\C = \code(\I)\intprod\code(\J)$ by \Cref{thm:product}.
It turns out that $\code(\I) = \{\mathbf{123}, \mathbf{24}, 12, 23, 1, 2, 4, \emptyset\}$ and $\code(\J) = \{\mathbf{1234}, 134, 12, 1, 2, \emptyset\}$ in this case, and \Cref{fig:decomp} shows this decomposition.
\begin{figure}[h]
    \[
    \includegraphics{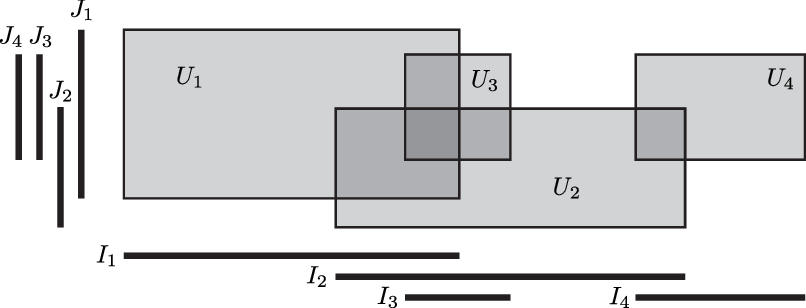}
    \]
    \caption{The boxes from \Cref{fig:intro-example}, decomposed into intervals in each coordinate direction.}
    \label{fig:decomp}
\end{figure}
\end{example}


\section{Convex codes that are not box-convex}\label{sec:small-codes}

Every box-convex code is also convex, but the converse fails even for codes on three indices.
Understanding the degree to which the converse fails helps us understand the difference between convex codes and box codes, and so we use \Cref{cor:classification} to characterize all codes on the index set $\{1,2,3\}$ that are convex, but not box-convex. 
There are three such codes, up to permuting the indices:
\begin{align*}
    \C_1 & = \{\mathbf{123}, 12, 13, 23, \emptyset\},\\
    \C_2 & = \{\mathbf{123}, 12, 13, 23,1, \emptyset\}, \text{ and}\\
    \C_3 & = \{\mathbf{12}, \mathbf{13}, \mathbf{23}, 1,2,3, \emptyset\}.
\end{align*}
The codes $\C_1$ and $\C_2$ are closed and open convex in $\R^2$ since they have only one maximal codeword (see \cite[Lemma 2.5]{local15}) and $\C_3$ is closed and open convex in $\R^2$ using three sets that pairwise intersect arranged in a triangular loop.
Arguing that all other convex codes on $\{1, 2, 3\}$  are box-convex requires us to treat a large number of cases by hand, which we do in an appendix following the bibliography. 
In this section, we will establish that the codes $\C_1, \C_2$, and $\C_3$ are not box-convex, and also give an interesting example of such a code on four indices (\Cref{thm:flaps}). 

The reason that $\C_1$ and $\C_3$ are not box-convex is that they have the following properties:
\begin{enumerate}
    \item[(i)] For $i=1,3$, the code $\C_i$ is not an interval code, and 
    \item[(ii)] For $i=1,3$, if $\C_i = \C\intprod \D$ where $\C$ and $\D$ are convex codes, then $\C = \C_i$ or $\D = \C_i$.
\end{enumerate}
Indeed, any code with these two properties is not box-convex---the first condition prevents it from being realized in $\R^1$, and the second condition means that it cannot be decomposed as an intersection product of interval codes since one of the factors would be the code itself. 
The code $\C_2$ requires a very slightly different argument, which we outline after our proof that $\C_1$ is not box-convex.\\
Intuitively, the second condition above alludes to a sense of irreducibility of box codes. 
\begin{definition}
    A code $\C$ is \emph{irreducible} if, whenever we can write $\C = \D\intprod \mathcal E$ for convex codes $\D$ and $\mathcal E$, we must have $\C = \D$ or $\C = \mathcal E$.
\end{definition}
\noindent Then, the above two conditions could be generalized into the following lemma: 
\begin{lemma}\label{lem:irreducible}
If a code $\C$ is irreducible and is not an interval code, then $\C$ is not box-convex. 
\end{lemma}
\begin{proof}
    Assume for the sake of contradiction that $\C$ is box-convex in $\R^d$. 
    Then, by \Cref{cor:classification}, $\C$ is the intersection product of $d$-many interval codes.
    Because $\C$ is irreducible, one of these interval codes must be $\C $ itself.
    However, this would imply that $\C$ is realizable in $\R^1$, violating our assumption that $\C$ is not an interval code. 
\end{proof}
Before moving on to our analysis of $\C_1$, $\C_2$, and $\C_3$, we establish a simple criteria that obstructs a code from being realized by intervals. 
\begin{lemma}\label{lem:notintervalcode}
Let $\C$ be a code with codewords $\sigma, \tau_1, \tau_2, \tau_3$ such that $\tau_i \subset \sigma$ and $\tau_i \not\subset \tau_j$ for all $i\neq j$, then $\C$ is not an interval code.
\end{lemma}
        
\begin{proof}

    Suppose for contradiction $\C$ is an interval code, with realization $\I$ in $\R^1$.
    Let $p$ be a point in the atom $\I^{\sigma}$. Let $q_i$ be a point in $\I^{\tau_i}$ for $1 \leq i \leq 3$. By pigeonhole principle, two of the $q_i$'s are on the same side of $p$. Without loss of generality, let $q_1 < q_2 < p$, so $q_1, q_2$ are both on the left side of $p$. 
    
    Since $\tau_1 \not\subset \tau_2$, there must be an index $j \in \tau_1$ such that $q_2 \not\in U_j$. 
    However, $q_1$ and $p$ are in $I_j$, and since $I_j$ is an interval, that means $q_2 \in I_j$ as well, a contradiction. 
    The figure below gives an illustration of this situation. 
\end{proof}
\begin{center}
            \begin{tikzpicture}
                \draw[gray, very thick] (0,0) -- (4,0);
                \draw[gray, very thick] (-3,0) -- (0,0);
                \draw[black, very thick] (-1.5,0) -- (2,0);
                
                \node at (2,0) {$\bullet$}; 
                \node at (2,0.3) {$p$};
                \node at (2,-0.4) {$\sigma$};
                \node at (-1.5,0) {$\bullet$}; 
                \node at (-1.5,0.3) {$q_1$};
                \node at (-1.5,-0.4) {$\tau_1$};
                \node at (-0,0) {$\bullet$}; 
                \node at (-0,0.3) {$q_2$};
                \node at (-0,-0.4) {$\tau_2$};
            \end{tikzpicture}
\end{center}

\begin{theorem}\label{thm:3neuroncode1}
The code $\C_1  = \{\mathbf{123}, 12, 13, 23, \emptyset\}$ is not box-convex. 
\end{theorem}
\begin{proof}
        By \Cref{lem:irreducible}, it suffices to establish that $\C_1$ is not an interval code, and that $\C_1$ is irreducible.\smallskip 
        
        \noindent \textbf{$\C_1$ is not an interval code.} This can be shown by letting $\sigma = 123, \tau_1 = 12, \tau_2 = 13, \tau_3 = 23$. 
        Note that all $\tau_i \subset \sigma$ and $\tau_i \not\subset \tau_j$ for all $i \neq j$.
        By \Cref{lem:notintervalcode}, $C_1$ is not an interval code. \\


\noindent \textbf{$\C_1$ is irreducible.} 
Suppose that $\C_1 = \C\intprod \D$ where $\C\subseteq 2^{\{1,2,3\}}$ and $\D\subseteq 2^{\{1,2,3\}}$ are convex codes. 
The maximal codeword $123$ in $\C_1$ is the intersection of codewords from $\C$ and $\D$ respectively, and the only way this can happen is if $123\in \C$ and $123\in \D$. 
Since every intersection of codewords between $\C$ and $\D$ lies in $\C_1$, we conclude that $\C\subseteq \C_1$ and $\D\subseteq \C_1$. 
Moreover, each of the codewords $\{12, 13, 23\}$ must appear in $\C$ or $\D$. 
By pigeonhole principle, two of these codewords appear in one of $\C$ or $\D$, and the third must appear in the same code lest we obtain a singleton codeword in $\C\intprod \D$. 
Hence, without loss of generality we have $\C = \{\mathbf{123}, 12, 13, 23, \emptyset\}$, which shows $\C_1$ that is irreducible. 
\end{proof}
 A very similar proof shows that $\C_2$ is not box-convex. Again, \Cref{lem:notintervalcode} shows that $\C_2$ is not an interval code.
The code $\C_2$ is not quite irreducible, but one can use a pigeonhole argument as above to show that if $\C_2 = \C\intprod \D$ then one of $\C$ and $\D$ is equal to $\C_2$ or $\C_1$.
In any case, $\C_2$ cannot be decomposed as an intersection product of interval codes. 
We now turn to $\C_3$.

\begin{theorem}
The code $\C_3  =\{\mathbf{12}, \mathbf{13}, \mathbf{23}, 1,2,3, \emptyset\}$ is not box-convex. 
\end{theorem}
\begin{proof}
\textbf{$\C_3$ is not an interval code.}
Suppose for contradiction that $\C_3$ has a realization $\I =\{I_1, I_2, I_3\}$ by intervals in $\R^1$.
The atoms of the maximal codewords $12$, $13$, and $23$ in this realization are each intervals. 
By permuting the indices, we can assume that they appear in the order $12$, $13$, $23$. 
But then $I_2$ contains points to the left and to the right of $\I^{13}$, and so $I_2$ covers $\I^{13}$.
Since the codeword $13$ does not contain 2, this is a contradiction. \smallskip

\noindent \textbf{$\C_3$ is irreducible.}
Suppose that $\C_3 = \C\intprod \D$ where $\C\subseteq 2^{\{1,2,3\}}$ and $\D\subseteq 2^{\{1,2,3\}}$ are convex codes. 
Since $123$ is not a codeword of $\C_3$, at least one of $\C$ or $\D$ does not have $123$ as a codeword.
Say without loss of generality that $\C$ does not contain $123$. 
Since $12$, $13$, and $23$ are codewords in $\C_3$, and each is the intersection of some codeword in $\C$ with a codeword in $\D$, each of $12$, $13$, and $23$ must appear as a codeword in $\C$. 
But since $\C$ is convex and does not contain $123$, this means that $\C$ must contain the singleton codeword $1$---indeed, any line segment from the atom of $12$ to the atom of $13$ in a convex realization of $\C$ must pass through the atom of $1$.
Similar reasoning shows that the singletons $2$ and $3$ are codewords of $\C$, so $\C=\C_3$. 
\end{proof}

We conclude with one last example of a convex code that is not box-convex, this time on four indices. 
Its realizations consist of a central region with three ``flaps'' attached, and it is not box-convex because one cannot add three flaps to a box to obtain a larger box.
Such ``flaps'' will also play a role in \Cref{sec:embedding-dimension-gap}. 

\begin{theorem}\label{thm:flaps}
The code $\C_4 = \{\mathbf{1234}, 12, 13, 14, \emptyset\}$ is open and closed convex in $\R^2$, but not box-convex in any dimension.
\end{theorem} 

\begin{proof}
The code $\C_4$ is open and closed convex in $\R^2$ since it has a unique maximal codeword (see \cite[Lemma 2.5]{local15}).
To see that $\C_4$ is not box-convex, we will use \Cref{lem:irreducible}.\smallskip

\noindent \textbf{$\C_4$ is not an interval code.} 
Let $\sigma = 1234, \tau_1 = 12, \tau_2 = 13, \tau_3 = 14$. Note that $\tau_i \subset \sigma$ and $\tau_i \not\subset \tau_j$ for all $i\neq j$.
By \Cref{lem:notintervalcode}, $C_1$ is not an interval code. \smallskip

\noindent \textbf{$\C_4$ is irreducible.} 
Suppose that $\C_4 = \C\intprod \D$ where $\C\subseteq 2^{[4]}$ and $\D\subseteq 2^{[4]}$ are convex codes. 
Since $1234$ is the intersection of codewords from $\C$ and $\D$, we must have $1234 \in \C$ and $1234 \in \D$.
This implies that $\C \subseteq \C_4$ and $\D \subseteq \C_4$.  \\
Since every pairwise intersection of elements in $\{12, 13, 14\}$ would produce the singleton codeword $1$, the only possibility is (up to swapping) $\C = \{ \mathbf{1234}, 12, 13, 14, \emptyset \}$ and $\D = \{\mathbf{1234}, \emptyset \}$. 
Thus $\C_4$ is irreducible. 
\end{proof}

\subsection{Non-monotonicity of box-convexity}\label{sec:non-monotone}
Cruz, Giusti, Itskov and Kronholm~\cite{CGIK} showed that open convexity is a monotone property of codes, in the sense that it is preserved by adding new non-maximal codewords.
Below, we use the notation $\Delta(\C) \eqdef \{\sigma\subseteq [n]\mid \sigma\subseteq c\text{ for some $c\in \C$}\}$ to denote the abstract simplicial complex generated by $\C$. 
\begin{theorem}[Monotonicity of open convexity, \cite{CGIK}] \label{thm:monotone}
Let $\C\subseteq 2^{[n]}$ be an open convex code, and let $\D\subseteq 2^{[n]}$ be a code with $\C\subseteq \D\subseteq \Delta(\C)$. 
Then $\D$ is open convex and $\odim(\D)\le \odim(\C)+1$.
\end{theorem}
This result fails for closed convex codes, as shown by Gambacini, Jeffs, Macdonald, and Shiu \cite{GJMS}. 
It is thus natural to ask whether an analogous result holds for box-convex codes.
It turns out that it does not. 
The method of proof used by Cruz, Giusti, Itskov, and Kronholm relies heavily upon the curvature of an open ball cutting through cylinders over a realization of $\C$.
Such a technique cannot be applied for boxes since they exhibit no curvature on the boundary.

The codes we constructed earlier in this section provide an explicit counterexample showing that box-convexity is not monotone. 
Consider the code $\C = \{\mathbf{1234}, 12, 13, \emptyset\}$.
This is certainly box-convex, and in fact this code can be realized by intervals in $\R^1$.
However, the code $\C_4$ of \Cref{thm:flaps} has the property that $\C\subseteq \C_4\subseteq \Delta(\C)$, but $\C_4$ is not box-convex.

\section{Open embedding dimension versus box embedding dimension}\label{sec:embedding-dimension-gap}

In this section, we explore differences in embedding dimensions between ordinary convex codes and box-convex codes.
The open embedding dimension and box embedding dimension of a code may be vastly different.
Identifying specific characteristics of these differences allows us to better understand restrictions on box-convexity.

Our examples in this section are based on codes whose realizations are sunflowers.
We say that $\U = \{U_1,\ldots, U_n\}$ is a \emph{sunflower} if $\code(\U) = \{ [n], 1, 2, \ldots, n, \emptyset \}$.
The sets $U_i$ are called \emph{petals} and $\bigcap_{i\in[n]} U_i$ is called the \emph{center} of the sunflower.
Sunflowers have proved to be highly useful in the study of open embedding dimensions of convex codes \cite{sunflowers, embeddingphenomena}, and they will additionally help us prove that box embedding dimension can be arbitrarily larger than the closed or open embedding dimensions of a code.
Specifically, we will prove the following theorem. 

\begin{theorem} For even $n\ge 4$, the code $\mathcal F_n \eqdef \{ [n], 1, 2, \ldots, n, \emptyset \}$ has $\bdim(\mathcal F_n) = n/2$, while on the other hand $\odim(\mathcal F_n) = \cdim(\mathcal F_n) =2$.
\end{theorem}
\begin{proof}
Note that $\odim(\mathcal F_n)$ and $\cdim(\mathcal F_n)$ are both at most two, since $\mathcal F_n$ has a unique maximal codeword (see \cite[Lemma 2.5]{local15}). 
To prove that $\bdim(\mathcal F_n) = n/2$, we work by induction, starting from the slightly smaller base case $n=2$. 

Let $n = 2k$, where $k\ge 1$.
We will prove that $\bdim(\mathcal F_n) = n/2$ by induction on $k$.  
In the base case $k=1$, observe that $\mathcal F_2 = \{\mathbf{12}, 1, 2, \emptyset\}$, which can be realized by intervals in $\R^1$. 
For the inductive step, it suffices to prove that $\mathcal F_n$ is the intersection product of $k$-many interval codes, and no fewer.
We have $\mathcal F_n = \C_1\intprod\cdots \intprod \C_k$, where $\C_i = \{[n], 2i-1, 2i, \emptyset\}$, which establishes that $\bdim(\mathcal F_n) \le k$.

To establish the matching lower bound, assume for contradiction that  $\mathcal F_n = \C_1 \intprod \cdots \intprod \C_{\ell}$ where $\ell < k$ and each $\C_i$ is an interval code.
Since $[n]$ is a codeword of $\mathcal F_n$, it must be a codeword of each $\mathcal C_i$.
Consequently, we have $\mathcal C_i\subseteq \mathcal F_n$ for each $1\le i \le \ell$, and moreover each singleton codeword must appear in some $\C_i$.
However, each $\C_i$ can contain at most two singleton codewords, since it is an interval code consisting of the codeword $[n]$ together with singletons and the empty set. 
There are $2k > 2\ell$ total singletons that must appear among the $\C_i$, a contradiction by pigeonhole principle. 
We conclude that $\mathcal F_n$ is the intersection product of $k$-many interval codes and no fewer, proving the result. 
\end{proof}

\section{Integer coordinates and open boxes}\label{sec:integer-coords}

In this section we discuss a strong application of the product theorem to box codes.
As mentioned in \Cref{sec:introduction}, the product theorem applies naturally to boxes since each box in $\R^d$ is the Cartesian product of $d$-many intervals. 
In particular, \Cref{thm:product} allows us to start with a collection of $n$ boxes in $\R^d$, decompose it into $d$-many collections of $n$ intervals in $\R^1$, perform some transformations on each collection of intervals in $\R^1$ that do not change their codes, and then return to a collection of boxes in $\R^d$ by taking products of our intervals.
In this way, we can obtain nicer realizations in $\R^d$ by performing simple transformations in $\R^1$. 

Our main result in this section is \Cref{thm:integercoordinates}, which guarantees realizations that are ``nicer" in the sense that all the boxes have corners with bounded integer coordinates, and interiors that realize the same code. 
Moreover, this theorem guarantees that the family of  codes realizable by open boxes (which are products of nonempty open intervals) is identical to the family of box-convex codes, i.e. there is no difference between codes realizable with open boxes and those realizable with closed boxes. 
As mentioned in \Cref{sec:introduction}, this contrasts the situation for convex codes, where open convex codes differ greatly from closed convex codes. 

\begin{theorem}\label{thm:integercoordinates}
Given a realization $\U = \{U_1, \ldots, U_n\}$ of a code $\C\subseteq 2^{[n]}$ by boxes in $\R^d$, there exists a realization $\V = \{V_1,\ldots, V_n\}$ of $\C$ by boxes, with the following properties:
\begin{itemize}
    \item[(i)]each $V_i$ has corners with integer coordinates between 1 and $2n$; that is, each $V_i$ is the Cartesian product of closed intervals with integer endpoints between $1$ and $2n$, and
    \item[(ii)] the interiors of the $V_i$ also comprise a realization of $\C$. 
\end{itemize}
Moreover, if $\C$ has a realization by interiors of boxes in $\R^d$, then it has a realization by boxes in $\R^d$. 
\end{theorem}
\begin{proof} 
First, we work with just one coordinate.
We start with a collection $\I = \{I_1, \ldots, I_n\}$ of closed intervals in $\R^1$. 
We may express each interval according to its endpoints, as $I_1 = [a_1, a_2]$, $I_2 = [a_3, a_4]$, and so on.
Observe that the relative order of these endpoints completely determines $\code(\I)$.
Specifically, for every $p\in \R$ we have $i\notin \patt_\I(p)$ if and only if $a_{2i} < p$ or $p < a_{2i-1}$.
The relative order of the endpoints also determines the code of the interiors of the intervals in $\I$, according to the fact that a point $p$ does not lie in the interior of $I_i$ if and only if $a_{2i} \le p$ or $p \le a_{2i-1}$.

To prove item (i) in the theorem, consider the function $f : \mathbb{R} \to \mathbb{N}$ which sends $a_i$ to $k$, where $a_i$ is the $k$-th smallest number out of the set of endpoints $\{a_1, a_2, \ldots, a_{2n}\}$.
Note that if two endpoints $a_i$ and $a_j$ are equal, then $f(a_i) = f(a_j)$.
So, $f$ sorts the endpoints of the intervals from least to greatest, and equal endpoints are sent to the same number.
Now, consider the collection of intervals $\J \eqdef \{J_1, \ldots, J_n\}$ where
\[
    J_i \eqdef [f(a_{2i-1}), f(a_{2i})] \text{\,   for $i\in [n]$}. 
\]
The collection $\J$ realizes the same code as $\I$ since the function $f$ preserves the relative order of the endpoints.
Moreover, $\J$ consists of intervals with integer endpoints between $1$ and $2n$, since $f(a_i)$ is always an integer between $1$ and $2n$. 
This proves condition (i) for $d=1$.

To prove (ii), consider the collection $\J' = \{J_1', \ldots, J_n'\}$ where $J_i'\eqdef [f(a_{2i-1}) - 1/4, f(a_{2i}) + 1/4]$. 
That is, we extend each interval in $\J$ by $1/4$ to the left and to the right.
This operation preserves the relative order of the endpoints of the intervals, except when a left endpoint coincides with a right endpoint, in which case these endpoints become separated with the new left endpoint strictly smaller than the new right endpoint.
The latter change merely widens the atom of the codeword that appeared where this left and right endpoint coincided, and so does not remove or introduce any codewords in the realization.
In particular, $\code(\J') = \code(\J)$.
See \Cref{fig:adjustment} for an example of this.

\begin{figure}[h] \[
\includegraphics{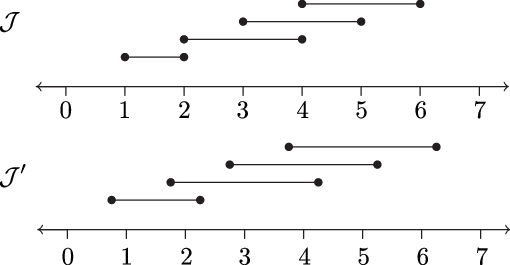}
\]
\caption{Extending the collection of intervals $\J$ to a collection $\J'$ whose closures and interiors both realize the code of $\J$.}
\label{fig:adjustment}
\end{figure}

We claim moreover that the interiors of the intervals in $\J'$ realize the same code as $\J'$.
Let $\J''$ denote the collection of interiors of intervals in $\J'$.
To see that $\code(\J') = \code(\J'')$, first note that $\patt_{\J'}(p) \neq \patt_{\J''}(p)$ only when $p$ is a left or right endpoint of an interval in $\J'$. 
Moreover, there are no points in $\R^1$ that are both a left and right endpoint of some interval in $\J'$. 
This implies that every intersection pattern that occurs in $\J'$ occurs at a point that is not an endpoint of any interval---if a pattern occurs at a left endpoint, it also occurs slightly to the right of that endpoint, and symmetrically for right endpoints.
Hence $\code(\J')\subseteq \code(\J'')$.
A symmetric argument shows that the intersection patterns that occur in $\J''$ also occur away from endpoints of the intervals, implying that $\code(\J'')\subseteq \code(\J')$.
Thus $\code(\J') = \code(\J'')$ as desired, proving condition (ii).

To obtain conditions (i) and (ii) simultaneously, we may simply apply the sorting argument used earlier to $\J'$, obtaining closed intervals with integer endpoints between $1$ and $2n$ and the same relative order of endpoints as in $\J'$.
Since the relative order of endpoints also determines the code of the interiors, this collection will have interiors that realize $\C$ as well. 

A symmetric argument shows that if $\C$ has a realization by open intervals then it has a realization by closed intervals.
Again, we can transform our starting realization so that the open intervals have integer endpoints.
We then shrink each interval by $1/4$ from the left and right, separating left and right endpoints, and noting that this does not affect the code of the intervals.
With left and right endpoints separated, the closures of these intervals will realize the same code as their interiors.
This proves the entirety of the theorem for $d=1$. 

For $d > 1$, note that a collection of $n$-many boxes $\U$ in $\R^d$ can be decomposed into collections $\I_1, \ldots, \I_d$ of intervals, where $\I_j$ consists of the $n$-many intervals obtained by projecting the boxes in $\U$ onto the $j$-th coordinate. 
\Cref{thm:product} implies that $\code(\U) = \code(\I_1)\intprod \cdots \intprod \code(\I_d)$.
Since we have established \Cref{thm:integercoordinates} for $d=1$, we can replace each $\I_j$ by a realization $\J_j$ consisting of intervals with integer endpoints between $1$ and $2n$ whose interiors realize the same code as $\I_j$. 
Then define $\V = \{V_1, \ldots, V_n\}$ where $V_i$ is the Cartesian product of the $i$-th interval from each $\J_j$. 
Again applying \Cref{thm:product} we have $\code(\V) = \code(\J_1)\intprod\cdots \intprod \code(\J_d)$, and since $\code(\J_j) = \code(\I_j)$ for all $j\in[d]$ this proves (i).
The fact that $\code(\J_j)$ is the same as the code of the interiors of the $\J_j$ proves (ii).
Finally, if we start with a collection of open boxes in $\R^d$, we may likewise decompose them into collections of open intervals, then apply \Cref{thm:integercoordinates} for $d=1$ to obtain realizations of the same codes by closed intervals, and finally apply \Cref{thm:product} to combine these collections of closed intervals to obtain a realization of the original code by closed boxes in $\R^d$.
\end{proof}

\section{Weak monotonicity of closed convexity}\label{sec:weak-monotonicity}

Our goal in this section is to prove \Cref{thm:weak-monotonicity}, which provides a weak analog of monotonicity of open convexity (\Cref{thm:monotone}) for closed convex codes. 
The result is ``weak'' in two senses.
First, the new non-maximal codewords that we add to our code must comprise a downward-closed set, instead of being arbitrary. 
This is necessary, as Gambacini, Jeffs, Macdonald, and Shiu~\cite{GJMS} gave an example of a closed convex codes with the property that adding a certain non-maximal codeword yields a code that is not closed convex in any dimension. 
Second, the closed embedding dimension may increase by two in our theorem, whereas the monotonicity of open convexity theorem established by Cruz, Giusti, Itskov, and Kronholm~\cite{CGIK} requires a dimension increase of at most one. 
It would be interesting to find examples of codes where our ``$+2$'' term is necessary, if they exist.

Below, a \emph{downward-closed set} is a set of elements in a poset which contains every element of the poset that is less than some element of the set. 
Also, recall from \Cref{sec:non-monotone} that $\Delta(\C)$ is the abstract simplicial complex generated by $\C$.

\begin{theorem}\label{thm:weak-monotonicity}
Let $\C\subseteq 2^{[n]}$ be a closed convex code.
Let $\D$ be a downward-closed subset of $\Delta(\C)\setminus \C$.
Then $\C\cup\D$ is closed convex, and moreover 
\[
\cdim(\C\cup\D) \le \cdim(\C) + 2. 
\]
\end{theorem}
\begin{proof}
By \Cref{cor:dimensions}, it suffices to show that $\C\cup\D$ is equal to $\C\intprod \mathcal E$, where $\mathcal E$ is a code with $\cdim(\mathcal E) \le 2$. 
In fact, the choice $\mathcal E \eqdef \D \cup \{[n], \emptyset\}$ will suffice.
Certainly we have $\cdim(\mathcal E)\le 2$ since $\mathcal E$ has a unique maximal codeword (see \cite[Lemma 2.5]{local15}).

To prove that $\C \cup \D = \C\intprod \mathcal E$, we show containment in each direction. 
We have $\C\subseteq \C\intprod \mathcal E$, since $[n]$ is a codeword of $\mathcal E$. 
Moreover, if $\tau\in \D$ then there exists a maximal codeword $\sigma$ of $\C$ with $\tau\subseteq \sigma$. 
We then have $\tau = \sigma \cap \tau$, which proves that $\tau\in \C\intprod \mathcal E$.

For the reverse containment, let $\sigma\in \C$ and $\tau \in \mathcal E$. 
If $\tau = [n]$, then $\sigma\cap \tau = \sigma$, which lies in $\C$ and hence in $\C\cup \D$.
If $\tau = \emptyset$, then $\sigma\cap \tau = \emptyset$, which again lies in $\C$. 
Otherwise, $\tau\in \D$. 
Note that $\C\cup\D$ contains $\Delta(\D)$ since $\D$ is a downward-closed subset of $\Delta(\C)\setminus \C$.
Thus $\C\cup \D$ contains $\sigma\cap \tau$, since $\sigma\cap \tau$ is a subset of $\tau$.
We conclude that $\C\cup \D = \C\intprod \mathcal E$, which proves the result. 
\end{proof}


\section{Conclusion and future questions}

We have developed a useful ``product theorem'' (\Cref{thm:product}) for realizations of codes, and used it to explain several features of box-convex codes, including points of difference compared to ordinary convex codes. 
However, a number of fundamental questions remain open, and we describe several below that may be of interest for future research. 

\begin{enumerate}
    \item For each fixed $d\ge 2$, is there a polynomial-time algorithm to recognize whether or not a given code $\C\subseteq 2^{[n]}$ has a realization by boxes in $\R^d$? 
     Every interval code $\C\subseteq 2^{[n]}$ has at most $2n+1$ codewords (see the proof of \cite[Proposition 1.5]{dimension1}), which together with \Cref{cor:classification} implies that every box-convex code in $\R^d$ has at most $(2n+1)^d = O(n^d)$ codewords.
    Thus it would make sense to measure the complexity of such an algorithm as a function of $n$. 
    Note that \cite{dimension1} has proved that there is an efficient algorithm for recognizing interval codes. 
    \item Among all box-convex codes $\C\subseteq 2^{[n]}$, which ones have $\bdim(\C)$ as large as possible? What is this largest value in terms of $n$?
    \item Can we classify ``local obstructions'' to box-convexity, specializing the notion of local obstructions for convex codes (see \cite{local15})? This would mean classifying which simplicial complexes are the nerve of a collection of boxes whose union is a box, analogous to the ``convex union representable" complexes of Jeffs and Novik \cite{CUR}. 
    \item Can we find further applications of \Cref{thm:product} to general open and closed convex codes?
    \item Curry, Jeffs, Youngs, and Zhao \cite{CJYZ} recently showed that every ``inductively pierced" code has a realization consisting of open or closed balls. 
    Does every inductively pierced code also admit a realization by boxes?
    \item Work of Goldrup and Phillipson \cite{GP-5neurons} classifies all open and closed convex codes on up to five indices. 
    Our work has only enumerated box-convex codes on up to three indices.
    This motivates us to ask: which convex codes on up to four or five indices are not box-convex?
    
\end{enumerate}

\bibliographystyle{plain}
\bibliography{convex_codes.bib}

\section*{Appendix}
This appendix supplements \Cref{sec:small-codes}, in which we showed that the codes $\C_1$, $\C_2$, and $\C_3$ are open and closed convex but not box-convex. 
Here we will argue that, up to permutation of indices, every other open or closed convex code $\C\subseteq 2^{\{1,2,3\}}$ is box-convex. 
We must first enumerate the open and closed convex codes $\C\subseteq 2^{\{1,2,3\}}$ up to symmetry, which we do in \Cref{fig:3indices}, partitioning them according to their maximal codewords.
The fact that this is the full list of open and closed convex codes on three indices follows from the fact that such a code must be max-intersection complete (i.e. contain all possible intersections of maximal codewords) and every max-intersection complete code $\C\subseteq 2^{\{1,2,3\}}$ has an open and closed convex realization (see \cite[Supplementary Text S1]{local15}).
In fact, we will do slightly more than show that these codes are box-convex; we will determine their exact box embedding dimensions. 

\begin{figure}[h]
    \begin{center}
    \begin{tabularx}{0.8\linewidth}{|XXX|}
    \hline 
    \multicolumn{1}{|l|}{$\{\mathbf{1},\emptyset\}$} &  \multicolumn{1}{|l|}{$\{\mathbf{1},\mathbf{2}, \emptyset\}$}&$\{\mathbf{1},\mathbf{2},\mathbf{3},\emptyset\}$\\
    \hline
    $\{\mathbf{12}, \emptyset\}$ & $\{\mathbf{12}, 1, \emptyset\}$ & $\{\mathbf{12}, 1,2,  \emptyset\}$\\\hline
    \multicolumn{1}{|l|}{$\{\mathbf{12},\mathbf{3}, \emptyset\}$}&
    \multicolumn{1}{|l|}{$\{\mathbf{12},\mathbf{13}, 1,  \emptyset\}$}&
    $\star\, \{\mathbf{12}, \mathbf{13}, \mathbf{23}, 1,2,3, \emptyset\}$\\
    \multicolumn{1}{|l|}{$\{\mathbf{12},\mathbf{3}, 1, \emptyset\}$}&
    \multicolumn{1}{|l|}{$\{\mathbf{12},\mathbf{13}, 1, 2, \emptyset\}$}&\\
    \multicolumn{1}{|l|}{$\{\mathbf{12},\mathbf{3}, 1, 2,\emptyset\}$}&
    \multicolumn{1}{|l|}{$\{\mathbf{12},\mathbf{13}, 1, 2, 3, \emptyset\}$}&\\
    \hline 
    $\{\mathbf{123}, \emptyset\}$ & $\{\mathbf{123}, 1, \emptyset\}$ & $\{\mathbf{123}, 1, 2, \emptyset\}$\\
    $\dagger\, \{ \mathbf{123}, 1, 2, 3, \emptyset\}$ & $\{\mathbf{123}, 12, \emptyset\}$ &$\{\mathbf{123}, 12, 1, \emptyset\}$ \\
    $\{\mathbf{123}, 12, 3, \emptyset\}$&$\{\mathbf{123}, 12, 1,2, \emptyset\}$& $\{\mathbf{123}, 12, 1, 3, \emptyset\}$\\
    $\dagger\, \{\mathbf{123}, 12, 1, 2,3, \emptyset\}$ &$\{\mathbf{123}, 12, 13, \emptyset\}$ &$\{\mathbf{123}, 12, 13, 1, \emptyset\}$ \\
    $\{\mathbf{123}, 12, 13, 2, \emptyset\}$ & $\{\mathbf{123}, 12, 13, 1, 2, \emptyset\}$&$\{\mathbf{123}, 12, 13, 2, 3, \emptyset\}$ \\
    $\dagger\, \{\mathbf{123},12, 13, 1,2,3, \emptyset\}$ & $\star\, \{\mathbf{123}, 12, 13, 23, \emptyset\}$& $\star\, \{\mathbf{123}, 12, 13, 23, 1, \emptyset\}$ \\
    $\dagger\, \{\mathbf{123}, 12, 13, 23, 1,2,\emptyset\}$ &  $\dagger\, \{\mathbf{123}, 12, 13, 23, 1,2,3, \emptyset\}$&\\
    \hline
    \end{tabularx}
    \end{center}
    \caption{All open and closed convex codes on the index set $\{1,2,3\}$, up to symmetry. 
    Those marked with $\star$ were shown not to be box-convex in \Cref{sec:small-codes}.
    Those marked with $\dagger$ have box embedding dimension equal to two.
    The remainder can be realized by intervals. }
    \label{fig:3indices}
\end{figure}





We leave it as an exercise to verify that the unmarked codes in \Cref{fig:3indices} can be realized by closed intervals, i.e. have box embedding dimension equal to one.
To prove that the $\dagger$-marked codes have box embedding dimension two, we require the following proposition.

\begin{proposition}\label{prop:bdim1}
Any code $\C \subseteq 2^{\{1,2,3\}}$ that contains the codewords  $123,1,2,$ and $3$ has $\bdim(\C) \ge 2$, i.e. it is not realizable with closed intervals in $\R^1$.
\end{proposition}

\begin{proof}
This follows from \Cref{lem:notintervalcode}, with $\sigma = 123$, $\tau_1 = 1$, $\tau_2 = 2$, and $\tau_3 = 3$. 
\end{proof}

This shows that all $\dagger$-marked codes from \Cref{fig:3indices} have box embedding dimension at least two, except the code $\{\mathbf{123}, 12, 13,23, 1, 2, \emptyset\}$. 
However, \Cref{lem:notintervalcode} likewise takes care of this case, with $\sigma = 123$, $\tau_1 = 12$, $\tau_2 = 13$, and $\tau_3 = 23$. 
Thus it remains to show that all $\dagger$-marked codes have realizations by boxes in $\R^2$, which we do below in \Cref{fig:final}. 

\begin{figure}[h]
    \[
    \includegraphics{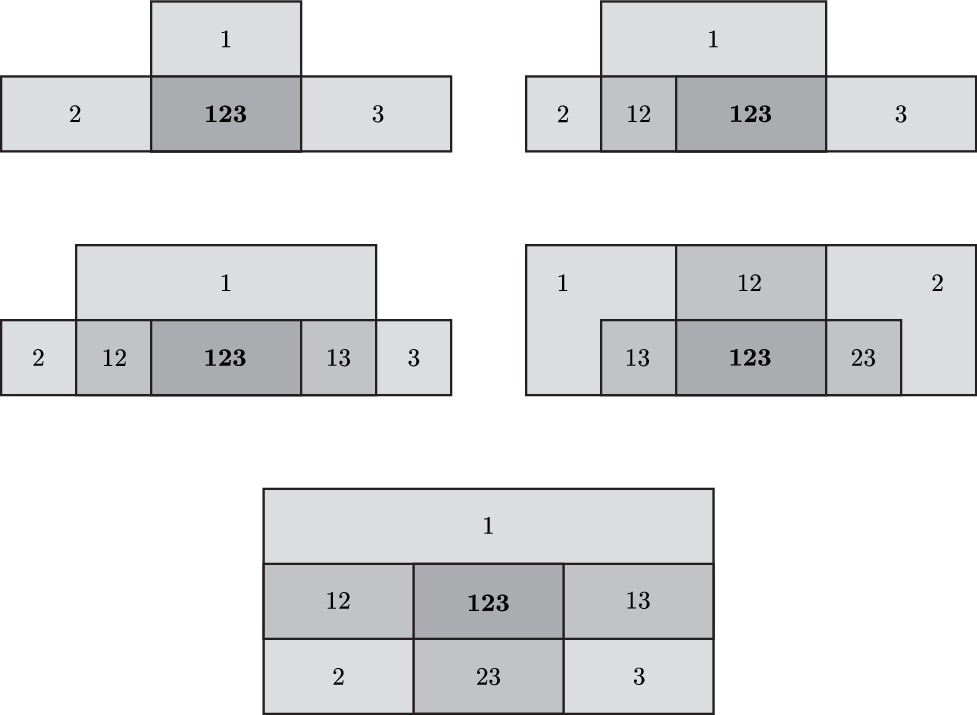}
    \]
    \caption{2-dimensional box realizations of the $\dagger$-marked codes from \Cref{fig:3indices}.}
    \label{fig:final}
\end{figure}


 


\end{document}